\documentclass[a4paper,11pt]{article}

\usepackage{longtable,geometry}
\usepackage{dsfont}
\usepackage[english]{babel}
\usepackage{amsmath}
\usepackage{amsfonts}
\usepackage{amssymb}
\usepackage{amsthm}
\usepackage{graphicx}
\usepackage{color}
\usepackage{hyperref}
\usepackage{enumerate}

\def\esp{\mathbb{E}}
\def\pr{\mathbb{P}}
\def\var{\mathbb{V}ar}
\def\N{\mathbb{N}}
\def\F{\mathcal{F}_n}
\def\R{\mathbb{R}}
\def\X{\boldsymbol{X}}

\def\x{\boldsymbol{x}}
\def\Y{\boldsymbol{Y}}
\def\y{\boldsymbol{y}}
\def\transpose{{^\intercal}}
\newtheorem{thm}{Theorem}[section]
\newtheorem{lemma}{Lemma}[section]

\newtheorem{remark}{Remark}[section]
\newtheorem{corollary}{Corollary}[section]
\newtheorem{proposition}{Proposition}[section]

\begin{document}

\title{Criterion for unlimited growth \\of critical
multidimensional stochastic models}

\author{Etienne Adam \thanks{CMAP, Ecole Polytechnique, CNRS, Universit\'e Paris-Saclay, route de Saclay, 91128 Palaiseau. France. Email : etienne.adam@polytechnique.edu}}
\maketitle
\begin{abstract}
We give a criterion for unlimited growth with positive probability for a large class of multidimensional stochastic models. As a by-product, we recover the necessary and sufficient conditions for recurrence and transience for critical multitype Galton-Watson with immigration processes and also significantly improve some results on multitype size-dependent Galton-Watson processes.\\
\end{abstract}

\tableofcontents

 \section{Introduction}

We study conditions on possible unlimited growth for sequences of random vectors $\X_n$, taking values in $\R_+^d$, which verify the stochastic difference equation 
\begin{equation}\label{rec}
\X_{n+1}=\X_nM+g\left(\X_n\right)+\xi_n \,,\, n\in \N,
\end{equation}
where $M$ is a non-negative primitive $d\times d$ matrix, $g: \R_+^d \rightarrow \R_+^d$ is a function such that $\|g(\x)\|=o(\|\x\|)$ when $\|\x\|$ tends to infinity, and $(\xi_n)$ is a sequence of random vectors (taking values in $\R^d$) such that almost surely
\[
\esp\left(\xi_n\big|\F\right)=0
\]
where $\{\F$, $n \in \N\}$ is the natural filtration associated to $\left(\X_n\right)$. We assume that $\X_0 \in \R^d_+$ and that random vectors $\xi_n$ are such that for all $n$, $\X_n$ takes values in $\R^d_+$ almost surely. 

The Perron-Frobenius Theorem \cite[pp. 3-4]{seneta} states that $M$ has a positive Perron root $\rho$ (which is also the spectral radius of $M$).
We call $\X_n$  ``subcritical'' if $\rho<1$, ``supercritical'' if $\rho>1$ and ``critical'' if $\rho=1$.
In the ``subcritical'' case, one has $\pr(\|\X_n\|\underset{n \rightarrow \infty}{\longrightarrow} \infty)=0$ because $\|\X_n\|$ is bounded in mean.
In many applications, one has $\pr(\|\X_n\|\underset{n \rightarrow \infty}{\longrightarrow} \infty)>0$ in the ``supercritical'' case.
This is well-known for the multitype Galton-Watson process with immigration, for instance. 
However, this is not necessarily the case in our general framework. For example, if $g(\X_n)=0$ and $\xi_n=\X_nM_n$ with $M_n$ independent and identically distributed random matrices such that $\pr(M_n=M)=\pr(M_n=-M)=1/2$,
then $\pr(\X_n\rightarrow 0)=1$.

In this article, we focus on the ``critical'' case, henceforth $\rho=1$. 
We define the normalized right and left eigenvectors $u$ and $v$ associated to $\rho$ in such a way
that $vu=u^\transpose u=1$.

We assume that the sequence $(\X_n)$ obeys a weak form of the Markov property. More precisely, we assume that $\esp\left(\left(\xi_nu\right)^2 \big|\F\right)$ is a function of $\X_n$ and will use the notation 
\[
\sigma^2\left(\X_n\right)=\esp\left(\left(\xi_nu\right)^2 \big|\F\right).
\]
The process $(\X_n)$ need not be a Markov chain because the law of $\xi_n$ may depend on $(\X_1, \X_2, \ldots, \X_n)$. However, all our examples are Markov chains.

The case $d=1$ is well understood. The interesting phenomenon is the fact that the growth is unlimited
depends on both the ``drift'' (\textit{i.e.} $g(\X_n)$) and the ``variance'' $\sigma^2\left(\X_n\right)$. 
This was first noticed by Lamperti \cite{lamperti} whose result was generalized by Kersting \cite{kersting}.
But, to the best of our knowledge, there is no criterion when $d>1$. 
Only particular examples were studied. For instance, Klebaner \cite{klebaner89,klebaner91}
gave sufficient conditions for unlimited growth or extinction for state-dependent multitype Galton-Watson processes. 
However, we can build some simple processes which do not satisfy his conditions. Gonzalez et al. gave also conditions for unlimited growth in the supercritical case in \cite{Gonzalez}. Jagers and Sagitov \cite{jagerssagitov} investigated population-size-dependent demographic processes that are particular cases of multidimensional growth models. Moreover in the critical case, they restricted themselves to bounded reproduction and bounded ``drift''.

The aim of this article is to obtain a criterion in any finite dimension that is analogous to the one in dimension one, which is our main result.
The strategy of the proof is the same as in Kersting's article \cite{kersting}. We shall illustrate our criterion with several
classes of examples, notably the one studied by Klebaner \cite{klebaner91} for which we get a complete picture (except for a very special case).

Under technical assumptions on functions $g$ and $\sigma^2$, we prove in this article that the process stays bounded a.s. if
\begin{equation}\label{crit1}
\limsup_{r \rightarrow +\infty} \frac{2rg\left(rv\right)u}{\sigma^2\left(rv\right)}<1,
\end{equation}
while it tends to infinity with positive probability if 
\begin{equation}\label{crit2}
 \liminf_{r \rightarrow +\infty} \frac{2rg\left(rv\right)u}{\sigma^2\left(rv\right)}>1. 
\end{equation}
This criterion is reminiscent of the criterion in Kersting's article \cite{kersting} in unidimensional models. 
In fact, the matrix $M$ preserves the component of $\X_n$ along the direction $v$ whereas it contracts along others directions.

In Section 2, we give our main result and its proof. We apply it in Section 3 to recover a recurrence-transience criterion for critical multitype Galton-Watson processes with immigration and to improve a criterion of almost-sure extinction for state-dependent multitype Galton-Watson processes. In the last section, we prove some lemmas which are used in the proof of Theorem 1.
\section{Criterion for unlimited growth}

\subsection{Assumptions}

For a row vector $\x$, let
\[
\y=\x\left(\mathrm{I}-uv\right).
\]
We assume that there exist a real number $\alpha$ such that $-1 < \alpha <1$, some positive real numbers $c_i$, $d_i$ and some real values functions $f_i$ and $h_i$ defined on $\R^d$, $i \in \{1,2\}$, such that
\[
\begin{cases}g\left(\x\right)u = c_1\left(\x u\right)^\alpha+h_1\left(\y\right)+f_1\left( \x \right) \\
\sigma^2\left(\x\right)=d_1 \left(\x u\right)^{1+\alpha}+h_2\left(\y\right)+f_2\left(\x \right) 
\end{cases}\tag{A1}\label{ineq}
\]
for all $\x \in \R _+ ^d$, with $h_1\equiv0$ if $\alpha \leq 0$ and 
\[
 \begin{cases} |h_1\left(\y\right)|\leq c_2\|\y\|^\alpha\\
  |h_2\left(\y\right)|\leq d_2 \|\y\|^{1+\alpha}\\
f_1\left(\x\right)=o\left(\left(\x u \right)^\alpha\right) \text{  when } \| \x \| \rightarrow +\infty\\
f_2\left(\x\right)=o\left(\left(\x u \right)^{1+\alpha}\right) \text{  when } \| \x \| \rightarrow +\infty,\\
 \end{cases}
\]
where $\| . \|$ stands for the euclidean norm.

We assume that there exist $\delta>0$ and $A_1>0$ such that for all $n \in \N$ and for all $\X_n \in \R^d_+$
\[
\esp\left(\left(\|\xi_n\|\right)^{2+\delta}\big|\F\right)\leq A_1\sigma^{2+\delta}\left(\X_n\right).\tag{A2}\label{moment2+}
\] 

We also need the following condition of unboundedness
\[
\forall C>0, \exists n\in \N \text{ such that }\pr\left(\X_nu\geq C\right)>0
\tag{A3}
\label{dichotomy}.
\]

Finally, we need two more assumptions on function $g$ and $\sigma^2$ to obtain possible unlimited growth for $\X_n$. Firstly, that $g(\x)u$ is bounded away from $0$:\\
There exists $s_1>0$ such that for all $a,b>0$ such that $s_1<a<b<\infty$, if $\x u \in (a,b)$ then $\exists \varepsilon>0$ such that 
\[
g(\x)u>\varepsilon,
\tag{A4}
\label{gnonzero}
\]
and secondly, that $\sigma^2$ is not infinite :
\[
 \forall a>0, \underset{\|\x\|<a}{\sup} \sigma^2(\x)<\infty.
\tag{A5}
\label{sigmafinite}
\]

\subsection{Main theorem}

We now give the criterion of unlimited growth for $\X_n$.
\begin{thm}[Unlimited growth criterion]
\leavevmode\\
 We assume \eqref{ineq} and \eqref{moment2+}.
\begin{enumerate}[i)]
\item If $c_1<\frac{d_1}{2}$ then
$\pr\left( \|\X_n\| \rightarrow +\infty\right)=0$.
\item If $c_1>\frac{d_1}{2}$ and \eqref{dichotomy}, \eqref{gnonzero} and \eqref{sigmafinite} hold then $\pr\left( \|\X_n\| \rightarrow +\infty\right)>0$.
\end{enumerate}
\end{thm}

Compared to \eqref{crit1} and \eqref{crit2}, we give a criterion in the special case where $g$ has a dominant term in $\left(\x u\right)^\alpha$. 
This may seem restrictive, nevertheless most of the applications deal with $\alpha=0$, 
which means that $g$ is bounded by a constant.
The case $c_1=\frac{d_1}{2}$ remains unexplored except for critical multitype Galton-Watson processes with immigration under some moment assumptions (see remark in Section 3).

\subsection{Proof of the theorem}
The strategy of the proof of the theorem consists in showing that there exist an integer $k$ and a real-valued function $L$ such that
\[
\mathbb{E}\left(L \left(\X_{n+k}u\right) | \mathcal{F}_n\right) \leq L \left(\X_nu\right)
\] 
when $\X_nu$ is larger than some constant. Then we build a supermartingale and proceed by using the martingale convergence theorem.

Before proving the theorem, we state two key lemmas providing us Lyapounov function. The proofs involve some technical computations and are deferred to Section \ref{section:lemma-proof}.
\begin{lemma}\label{loglyapu}
Let us assume (A1) and (A2). If $c_1<\frac{d_1}{2}$, then there exists $s>0$ and $k\in \N^*$ such that,
\[
\esp\left(\log\left(\X_{n+k}u\right)\big|\F\right)\leq \log\left(\X_nu\right), \text{ if } \X_nu>s.
\]
\end{lemma}

\begin{lemma}\label{1/loglyapu}
Let us assume (A1) and (A2). If $c_1>\frac{d_1}{2}$, then there exists $s>0$ and $k \in \N^*$ such that, 
\[
\esp\left(L\left(\X_{n+k}u\right) \big| \F \right)  \leq L\left(\X_nu\right), \text{ if } \X_nu>s,
\]
with $L\left(x\right)=\left(\log x\right)^{-1}$.
\end{lemma}

\begin{proof}[Proof of Theorem 1]
\leavevmode\\
Without loss of generality, we assume that  for every $n \in \N$, $\X_nu \geq 3$ almost surely (otherwise consider $\X_n+3v$ instead of $\X_n$).

i) We start by considering the case where $c_1<\frac{d_1}{2}$.

Following \cite{kersting}, let us assume that $\X_nu \rightarrow +\infty$ with positive probability. Let $U_n=\X_{nk}u$, then $U_n \rightarrow +\infty$ with positive probability, too. Thus there is a positive integer $T$ such that 
\begin{equation}\label{contradiction}
\pr\left(\underset{n \geq T}{\inf} U_n > s, U_n \rightarrow +\infty\right)>0.
\end{equation}
Let $\tau=\inf \{ n\geq T : U_n \leq s \}$ with the convention that $\tau=+\infty$ if $\underset{n\geq T}{\inf} U_n > s$. Let
\[
V_n=\begin{cases}
\log\left(U_{n+T}\right) \text{ if }n+T \leq \tau,\\
\log\left(U_{\tau}\right) \text{ otherwise.} 
\end{cases}
\]
Since $(V_n)$ is a positive supermartingale by Lemma \ref{loglyapu}, it converges almost surely and we obtain a contradiction with \eqref{contradiction}.

ii) We now turn to the case where $c_1>\frac{d_1}{2}$.

Let $s>0$ be large enough, such that the statement of Lemma \ref{1/loglyapu} holds true. Let $A,B,C$ and $D$, four sets defined as follows :

\begin{itemize}
\item $A=\Big\{  \underset{n \rightarrow \infty}{\limsup} \X_nu \leq s\Big\}$
\item $B=\Big\{ \underset{n \rightarrow \infty}{\limsup} \X_nu<\infty \text{ and } \exists i \in \left\{0, \ldots, k-1\right\}$, $\X_{nk+i}u \underset{n \rightarrow \infty}{\longrightarrow} R_i \text{ with } s<R_i<\infty \Big\}$
\item $C=\Big\{ \exists i \in \left\{0, \ldots, k-1\right\}$, $\X_{nk+i}u \underset{n \rightarrow \infty}{\longrightarrow} \infty 
\text{ and }\underset{n \rightarrow \infty}{\limsup}\X_{nk+i+1}u<\infty\Big\}$
\item $D=\Big\{ \X_{n}u \underset{n \rightarrow \infty}{\longrightarrow} \infty\Big\}$.
\end{itemize}

We want to prove that $\pr(D)>0$. We first prove that $\pr(A \cup B \cup C \cup D)=1$ (Step 1). Secondly that $\pr(A)<1$ (Step 2). Thirdly that $\pr(B)=0$ (Step 3) and we conclude by proving that $\pr(C)=0$ (Step 4).

Step 1 : Let $i \in \left\{0, \ldots, k-1\right\}$. Then $U_{i,n}=\min \left(L\left(\X_{nk+i}u\right),L\left(s\right)\right)$ is a non-negative bounded supermartingale which converges almost surely and in mean by Lemma \ref{1/loglyapu}. Therefore, either $\X_{nk+i}u$ converges to a number greater than $s$, possibly $\infty$, or $\underset{n \rightarrow \infty}{\limsup} \X_{nk+i}u \leq s$. So, $\pr(A \cup B \cup C \cup D)=1$.

Step 2 : Let us assume that $\pr\left(A\right)=1$. Then, for all $i \in \left\{0, \ldots, k-1\right\}$, $U_{i,n}$ converges to $L\left(s\right)$. Since $\esp\left(\min \left(L\left(\X_{n+kl}u\right),L\left(s\right)\right)\right)\leq \esp\left(\min \left(L\left(\X_{n}u\right),L\left(s\right)\right)\right)$, for all $l\in \N$, we obtain
\[
 \esp\left(\min \left(L\left(\X_{n}u\right),L\left(s\right)\right)\right)\geq L\left(s\right).
\]
But by definition, $\min \left(L\left(\X_{n}u\right),L\left(s\right)\right)\leq L\left(s\right)$, therefore $\min \left(L\left(\X_{n}u\right),L\left(s\right)\right)=L\left(s\right)$ a.s. or $\X_{n}u\leq s$ for all $n$, which contradicts \eqref{dichotomy}.

Thus, $\pr\left(B \cup C \cup D\right)>0$.

Step 3 : Without loss of generality, let us assume $B$ with $i=0$ and let $R=2\sup \X_nu$.

By \eqref{gnonzero}, there exists $s_1<R_0$ and $\varepsilon>0$ such that for all $\x$ such that $\x u\in (s_1,R)$, $g\left(\x\right)u >\varepsilon$.

Since $\X_{nk}u$ converges to $R_0$, there exists $N_0$ such that  $\forall n\geq N_0$, $\X_{nk}u \in \left(s_1,R\right)$.

Thus, we can choose $N$ such that $\forall n\geq N$, $\X_nu \leq R$ and $g\left(\X_{N}\right)u \neq 0$.

Consider now,
\[
A_n=\sum_{l=0}^n g\left(\X_{N+l}\right)u, 
\]
\[
 M_n=\sum_{l=0}^n A_l^{-1} \xi_{N+l}u.
\]
One can check that $M_n$ is a martingale.

Further, by \eqref{gnonzero} and \eqref{sigmafinite},
\begin{align*}
 \sum_{n=0}^\infty \esp\left(\left(M_{n+1}-M_n\right)^2 \big| \mathcal{F}_{N+n+1}\right) & = \sum_{n=1}^\infty A_{n}^{-2}\sigma^2\left(\X_{N+n}\right)\\
& \leq C \sum_{n=1}^\infty A_{n}^{-2}\\
& \leq C' \sum_{n=1}^\infty A_{n}^{-2}\left(A_{n}-A_{n-1}\right)\\
& \leq C' \int_{A_1}^\infty t^{-2}dt<\infty.
\end{align*}
By a martingale convergence theorem, $M_n$ converges almost surely. Since $\X_{nk}u$ converges to $R$, $A_n \rightarrow +\infty$, and by Kronecker's lemma
\[
 \sum_{l=1}^n\xi_lu=o\left(A_n\right).
\]
For $n\geq N$, we get the contradiction :
\[
 \X_{n+1}u=\X_Nu+\sum_{l=N}^ng\left(\X_l\right)u+\sum_{l=N}^n \xi_lu=A_n+o\left(A_n\right) \rightarrow +\infty.
\]

Thus $\pr(B)=0$, so with positive probability, there exists $i \in \left\{0, \ldots, k-1\right\}$ such that $\X_{nk+i}u$ tends to infinity. Let us prove that this implies that $\X_nu$ tends to infinity.

Step 4 : Without loss of generality, let us assume that $\X_{nk}u$ tends to infinity.\\
Let $\gamma=\frac{2}{1-\alpha}$, $r>s$ and $n$ such that $n^\gamma>2r$.\\
Let $\Gamma_n=[n^\gamma,(n+1)^\gamma)$ be a sequence of intervals and $N \left(l\right)$ be an increasing sequence of stopping times defined by
\[
 N(l)=\inf \{ n\> N(l-1) \text{ such that } \X_{nk}u \in [l^\gamma, \infty)\}.
\]
By Markov's inequality, we obtain
\begin{align*}
 \pr\left(\X_{N\left(n\right)k+1}u\leq 2r \big| \X_{N\left(n\right)k}u\geq n^\gamma\right) &  \leq \underset{l\geq n}{\sup} \, \, \pr\left(\xi_{N(n)k}u \leq \left(2r-l^\gamma\right)\big| \X_{N\left(n\right)k}u \in \Gamma_l\right)\\
& \leq  \underset{l\geq n}{\sup} \, \, \pr\left(\left(\xi_{N\left(n\right)k}u\right)^2 \geq \left(l^\gamma-2r\right)^2 \big| \X_{N\left(n\right)k}u \in \Gamma_l\right)\\
& \leq \underset{l\geq n}{\sup} \, \, \frac{Kl^{\gamma\left(\alpha+1\right)}}{l^{2\gamma}}\\
& \leq K' \frac{1}{n^2}.
\end{align*}

Hence by Borel-Cantelli lemma, for any $r>0$ sufficiently large, 
\[
\pr\left(\underset{n \rightarrow \infty}{\limsup} \X_{nk+1}u \leq r \big| \X_{nk}u\rightarrow \infty\right)=0.
\]
Thus, $\X_{nk+1}u$ converges to infinity so $\pr(C)=0$.\\
Since $\pr\left(A \cup B\cup C \cup D\right)=1$ and $\pr(B)=\pr(C)=0$, thus $\pr(A \cup D)=1$. Since $\pr(A)<1$, we have the desired result 
\[
\pr(D)=\pr\left(\X_{n}u\rightarrow \infty\right)>0. 
\]
\end{proof}

\section{Applications}

Our applications focus on the case $\alpha=0$. 
This is because we consider population models with finite variance of number of offsprings per individual. 
Thus, $\sigma$ has to be of the order of $\x u$ and $g$ of the order of a constant.
Notice also that all models here are Markov chains, although our result is applicable to processes that need not be Markov chains. 
In the particular case of irreducible Markov chains, the process has an unlimited growth with positive probability if and only if the chain is transient. Conversely, it does not tend to infinity a.s. if and only if the chain is recurrent.

\subsection{Multitype Galton-Watson process with immigration}

A first class of processes governed by the stochastic difference equation \eqref{rec} is given by critical multitype Galton-Watson processes with immigration. Kawazu \cite{kawazu} gave a criterion of recurrence and transience that he proved by using generating functions. We recover here the same result.

Let $\left(Z_n\right)$ be a critical multitype Galton-Watson process with immigration with $d$ types.
At generation $n$, the $k$-th individual of type $i$, $i \in \{1, \ldots, d\}$ and \\$k \in \{1, \ldots, \left(Z_n\right)_i\}$, gives birth to $X_{i,j,k,n}$ individuals of type $j$,
$j \in \{1, \ldots, d\}$.
The random vectors $(X_{i,j,k,n})_{j \in \{1, \ldots, d\}}$ with $i \in \{1, \ldots, d\}$, $k \geq 1$ and $n \in \N$ are independent with distribution depending only on $i$.
To alleviate notations, we write $X_{i,j}$ for $X_{i,j,1,1}$.

We assume that for all $i$, $j \in \{1, \ldots, d\}$, $\pr(X_{i,j}=0)>0$.
Let $M=\big(\esp(X_{i,j})\big)_{i,j}$ be the mean matrix. 
We assume that $M$ is a non-negative primitive matrix.
Since the process is critical, the largest eigenvalue of $M$ is $1$. 
Let $u$ (resp. $v$) the right (resp. the left) eigenvector corresponding to this eigenvalue.
At each generation $n$, $A_n \in \N^d$ individuals immigrate.
The random variables $A_n$ are independent and identically distributed, with $\pr(A_1=(0, \ldots,0))>0$, $\esp(A_1)=a$ and $\var(A_1u)=\tau^2$.
The random variables $A_n$ are also independent of all variables $X_{i,j,k,n}$.
Therefore we have $\esp(Z_{n+1})=\esp(Z_n)M+a$.

We assume that there exists $\delta>0$ such that  for $\left(i,j\right) \in \{1, \ldots, d\}^2$, 
\begin{equation}\label{hypo1}
\esp(X_{i,j}^{2+\delta})<+\infty 
\text{ and }
\esp\left(\left(A_1u\right)^{2+\delta}\right)<+\infty. 
\end{equation}
Let $\Gamma_i=\left(\text{Cov}\left(X_{i,j},X_{i,j'}\right)\right)_{j,j' \in \{1, \ldots, d\}}$ be the matrix of the covariances of offspring distributions.
Let 
\[
V(z)=\sum_{i=1}^{d}z_i\Gamma_i, 
\]
for $z \in \R^d$.
We obtain the stochastic difference equation 
\[
Z_{n+1}=Z_nM+a+\xi_n
\]
with\[
\xi_n=\left(\sum_{i=1}^{d}\sum_{k=1}^{\left(Z_n\right)_i}\left\{X_{i,j,k,n}-\esp\left(X_{i,j,k,n}\right)\right\}+A_ne_j-\esp\left(A_ne_j\right)\right)_{1\leq j \leq d}
\]
where the $(e_j)_{j \in \{1,\ldots,d\}}$ are the standard unit vectors and
\[
\esp((\xi_nu)^2 \big|\F)=u^\transpose V(Z_n)u+\tau^2.
\]
\begin{proposition}
The process $(Z_n)$ is 
\begin{itemize}
 \item recurrent if $2au<u^\transpose V(v)u$
\item  transient if $2au>u^\transpose V(v)u$.
\end{itemize}
\end{proposition}
\begin{remark}
Kawazu \cite{kawazu} obtained the same criterion under weaker assumptions: he did not require $\var(A_1u)<+\infty$ and \eqref{hypo1}. He also proved that the process is null recurrent when $2au=u^\transpose V(v)u$ if $\esp\left(X_{i,j}^2 \log \left(X_{i,j}\right)\right)<+\infty$ and \mbox{$\esp\left(A_1 \log \left(A_1\right)\right)<+\infty$}.
\end{remark}
\begin{proof}
Firstly, note that
\[
 \esp((\xi_nu)^2 \big|\F)=u^\transpose V(Z_n)u+\tau^2=(Z_nu)u^\transpose V(v)u+u^\transpose V(Z_n(I-uv))u+\tau^2,
\]
then recurrence and transience depend on the sign of $2au-u^\transpose V(v)u$.\\
Since \eqref{ineq} is verified with $\alpha=0$, $c_1=au$, $h_1=0$, $d_1=u^\transpose V(v)u$, $h_2(\y)=u^\transpose V(\y)u$, $f_1=0$ and $f_2=\tau^2$, and \eqref{dichotomy}, \eqref{gnonzero} and \eqref{sigmafinite} are also verified, we just have to check \eqref{moment2+} to apply Theorem 1.

Let $l \in \N^*$ and $\left(U_k\right)_{k \in \{1, \ldots, l\}}$ be some random variables independent with zero mean and such that $\esp\left(|U_k|^{2+\delta}\right)<+\infty$ for all $k \in \{1, \ldots, l\}$.
We can apply both Marcinkiewicz-Zygmund \cite[p. 108]{LinBai} and H\"older inequalities, \textit{i.e.} there exists $R>0$ such that
\[
 \esp\left(\left(\sum_{k=1}^{l}U_k\right)^{2+\delta}\right)  \leq R ~\esp\left(\left(\sum_{k=1}^{l}U_k^2\right)^{1+\frac{\delta}{2}}\right) \leq R l^{\frac{\delta}{2}}\esp\left(\sum_{k=1}^{l}U_k^{2+\delta}\right).
\]
Since there are three sums in $\|\xi_n\|$, we now apply three times the latter inequality to verify that \eqref{moment2+} holds:
\begin{align*}
 & \esp(\|\xi_n\|^{2+\delta} \big|\F)\\
& \leq 2^{2+\delta} \esp\left(\left(\sum_{j=1}^d\left(  \sum_{i=1}^{d}\sum_{k=1}^{\left(Z_n\right)_i}\left\{X_{i,j,k,n}-\esp\left(X_{i,j,k,n}\right)\right\}\right)^2\right)^{1+\frac{\delta}{2}} \big| \F \right)\\
& \quad + 2^{2+\delta} \esp\left(\|A_n-a\|^{2+\delta} \big| \F \right)\\
& \leq  2^{2+\delta}d^{\frac{\delta}{2}} \esp\left(\left(\sum_{j=1}^d\Big|\sum_{i=1}^{d}\sum_{k=1}^{\left(Z_n\right)_i}\left\{X_{i,j,k,n}-\esp\left(X_{i,j,k,n}\right)\right\}\Big|^{2+\delta}\right)\big| \F \right)\\
& \quad + 2^{2+\delta} \esp\left(\|A_n-a\|^{2+\delta} \big| \F \right)\\
& \leq R d^{\delta} \esp\left(\sum_{j=1}^d\sum_{i=1}^{d}\Big|\sum_{k=1}^{\left(Z_n\right)_i}\left\{X_{i,j,k,n}-\esp\left(X_{i,j,k,n}\right)\right\}\Big|^{2+\delta} \big| \F\right)\\
& \quad +2^{2+\delta} \esp\left(\|A_n-a\|^{2+\delta} \big| \F \right)\\
& \leq R^2 d^{\delta} \esp\left(\sum_{j=1}^d\sum_{i=1}^{d}\left(Z_n\right)_i^{\frac{\delta}{2}}\sum_{k=1}^{\left(Z_n\right)_i}|X_{i,j,k,n}-\esp\left(X_{i,j,k,n}\right)|^{2+\delta}\big| \F \right)\\
& \quad+2^{2+\delta} \esp\left(\|A_n-a\|^{2+\delta} \big| \F \right).
\end{align*}
We now apply \eqref{hypo1} to obtain
\[
\esp(\|\xi_n\|^{2+\delta} \big|\F) \leq C \left(\sum_{i=1}^{d}\left(Z_n\right)_i^{\frac{2+\delta}{2}}\right)+D \leq C' \sigma^{2+\delta}(Z_n),
\]
for $\|Z_n\|$ sufficiently large.
\end{proof}

\subsection{State-dependent multitype Galton-Watson processes}
State-dependent Galton-Watson processes were first introduced by Klebaner in \cite{klebaner84} and H\"opfner in \cite{hopf}. H\"opfner compared the probability generating functions of these processes with those of critical Galton-Watson processes with immigration to obtain a criterion of extinction. However, this idea seems difficult to be transfered to the multitype case.
Basically this is because we have to alter the transitions of the Galton-Watson with immigration process for an infinite number of states and thus we may change the nature of the process (recurrent or transient). 
Klebaner \cite{klebaner89,klebaner91} defined multitype state-dependent Galton-Watson processes for which he only gave sufficient conditions for extinction. In particular, he could not treat some range of parameter. In this subsection, we obtain a criterion to infer whether there is almost-sure extinction or survival with positive probability (except in a very special case).

Following \cite{klebaner91}, we define a discrete-time state-dependent multitype Galton-Watson process with $d$ types $Z_n$ by 
\[
 Z_{n+1}=\left(\sum_{i=1}^{d}\sum_{k=1}^{\left(Z_n\right)_i}X_{i,j,k,n}\left(Z_n\right)\right)_{j \in \{1, \ldots, d\}}
\]
where $X_{i,j,k,n}\left(z\right)$ is the number of type $j$ offspring of the $k$th type $i$ parent when the process is in the state $z$ in time $n$.
Given $Z_n=z$, the $k$th parent of type $i$ has a random vector of offspring 
\[
\left(X_{i,1,k,n}\left(z\right), \ldots, X_{i,d,k,n}\left(z\right)\right),\,k=1, \ldots, z_i.
\]
For each $n \in \N$, the offspring vectors of distinct parents $\left(k=1, \ldots, z_i, i=1, \ldots, d\right)$ are independent. Moreover, for a fixed parental type $i$, the offspring vectors are identically distributed for all $n$ and $k$, with distribution depending at most on the state $z$. 
For the sake of notation clarity, we write $X_{i,j}$ for $X_{i,j,1,1}$.
Let 
\[
M(z)=\left(\esp \left(X_{i,j}\left(z\right)\right)\right)_{i,j \in \{1, \ldots, d\}}
\]
 be the mean matrix.

We assume that
\[
M(z)=M+C(z)
\]
where $M$ is a non-negative primitive matrix with Perron root $1$ and corresponding right and left eigenvectors $u$ and $v$, with $vu=u^\transpose u=1$, and $C(z)$ is a non-negative matrix and we let 
\[
g(z)=zC(z).
\]
We assume that 
\[
 \lim_{\|z\|\rightarrow +\infty}g(z)=D \in \R_+^d.
\]
Let $\Gamma_i(z)=\left(\text{Cov}\left(X_{i,j}\left(z\right),X_{i,j'}\left(z\right)\right)\right)_{j,j' \in \{1, \ldots, d\}}$ be the matrix of the covariances of offspring distributions when the population size is in the state $z$. We assume that for all $i \in \{1, \ldots, d\}$, $\Gamma_i(z)$ converges to $\Gamma_i$ when $\|z\|$ converges to infinity.

Let 
\[
\tilde{V}(z)=\sum_{i=1}^{d}z_i\Gamma_i\left(z\right) 
\]
 be the conditional dispersion matrix of the next generation when the population is in the state $z$. We introduce also the quantity 
\[
V(z)=\sum_{i=1}^{d}z_i\Gamma_i. 
\]
Then $(Z_n)$ satisfies the stochastic difference equation 
\[
 Z_{n+1}=Z_nM+g\left(Z_n\right)+\xi_n,
\]
with 
\[
\xi_n=\left(\sum_{i=1}^d\sum_{k=1}^{\left(Z_n\right)_i}\{X_{i,j,k,n}\left(Z_n\right)-\esp\left(X_{i,j,k,n}\left(Z_n\right)\right)\}\right)_{j \in \{1, \ldots, d\}}.
\]
One can easily check that $\esp\left(\xi_nu \big| \F \right)=0$ and 
\[
\esp\left(\left(\xi_nu\right)^2 \big| \F \right)=u^\transpose \tilde{V}(Z_n)u= (Z_nu)u^\transpose V(v)u+u^\transpose V(Z_n(\mathrm{I}-uv))u+f(Z_n),
\]
where the function $f$ is such that $f(x)=o(\|x\|)$ when $\|x\|$ tends to infinity.\\
We assume that there exist $\delta>0$ and $K>0$ such that for all $i$, $j \in \{1, \ldots, d\}$ and $z \in \R^d_+$, 
\[
\esp\left(X_{i,j}\left(z\right)^{2+\delta}\right)<K.
\]
As in the previous example, the assumption \eqref{moment2+} is a consequence of Marcinkiewicz-Zygmund and H\"older inequalities.

We make the usual assumptions when one has in mind a population process: $0$ is an absorbing state and all states in $\N^d \setminus \{0\}$ communicate.
\begin{thm}
\leavevmode\\
If 
\[
\frac{2Du}{u^\transpose V(v)u}<1
\]
then the process becomes extinct almost surely.\\
If 
\[
 \frac{2Du}{u^\transpose V(v)u}>1
\]
then the process survives with positive probability.
\end{thm}

We cannot treat the case $\frac{2Du}{u^\transpose V(v)u}=1$.

We now illustrate this result by the following example.
\medskip

\noindent\textbf{Example.}
We take the example of a two-type cell division process from \cite{klebaner91}. We recall that $X_{i,j,k,n}(z)$ is the number of children of type $j$ for the $k$-th parent of type $i$ at generation $n$ when the population is at state $z$. Again, we write $X_{i,j}(z)$ for $X_{i,j,1,1}(z)$. 

We assume that $X_{i,j}(z)$ take values $0$ or $1$ with probabilities $p_{i,j}(z)$ and that $\pr(X_{i,1}(z)=0, X_{i,2}(z)=0)>0$, $i \in \{1,2\}$. Let $b_i(z)=\pr(X_{i,1}(z)=1, X_{i,2}(z)=1)$, $i \in \{1,2\}$ and $a_{i,j}(x)$, $i,j \in \{1,2\}$, be arbitrary functions non-vanishing for $x>0$, such that 
\[
 M(z)=\begin{pmatrix}p & 1-p\\ p' & 1-p'\end{pmatrix}+\begin{bmatrix}\dfrac{\strut c_1a_{1,1}(z_1)}{\strut z_1a_{1,1}(z_1)+z_2a_{2,1}(z_2)} & \dfrac{\strut c_2a_{1,2}(z_1)}{\strut z_1a_{1,2}(z_1)+z_2a_{2,2}(z_2)} \\ \dfrac{\strut c_1a_{2,1}(z_2)}{\strut z_1a_{1,1}(z_1)+z_2a_{2,1}(z_2)} & \dfrac{\strut c_2a_{2,2}(z_2)}{\strut z_1a_{1,2}(z_1)+z_2a_{2,2}(z_2)} \end{bmatrix}
\]
where $p,p' \in \left(0,1\right)$ and $c_1,c_2>0$. We assume that $b_i(z)\sim b_i$ when $\|z\|$ tends to infinity.\\
With the previous notations, we have 
\begin{itemize}
\item $u=\frac{1}{\sqrt{2}}\begin{pmatrix} 1 \\ 1 \end{pmatrix}$ and $v=\sqrt{2}\begin{pmatrix} \frac{p'}{1-p+p'} & \frac{1-p}{1-p+p'}\end{pmatrix}$
\item $Du=\frac{c_1+c_2}{\sqrt{2}}$
\item $V(z)=z_1 \begin{pmatrix} p(1-p) & b_1-p(1-p) \\ b_1-p(1-p) & p(1-p)\end{pmatrix} + z_2 \begin{pmatrix} p'(1-p') & b_2-p'(1-p') \\ b_2-p'(1-p') & p'(1-p')\end{pmatrix}$
\end{itemize}
\begin{corollary}
If
\[
c_1+c_2<\frac{p'}{1-p+p'}\, b_1+\frac{1-p}{1-p+p'}\, b_2
\]
then the process becomes extinct almost-surely.\\
If
\[
c_1+c_2>\frac{p'}{1-p+p'}\, b_1+\frac{1-p}{1-p+p'}\, b_2
\]
then the process survives with positive probability.
\end{corollary}
Klebaner in \cite{klebaner91} proved almost sure extinction if $c_1+c_2 < \min (b_1,b_2)$ and survival with positive probability if $c_1+c_2> \max(b_1,b_2)$. We thus have improved his result since we prove that the critical value for $c_1+c_2$ is $\frac{p'}{1-p+p'}b_1+\frac{1-p}{1-p+p'}b_2$. Except for the equality case, we get a complete picture of the fate of the process.

\section{Proof of Lemmas \protect{\ref{loglyapu}} and \protect{\ref{1/loglyapu}}}\label{section:lemma-proof}

In this section, we prove Lemmas \ref{loglyapu} and \ref{1/loglyapu}. The proof is based upon the following result.
Let
\[
\Y_n=\X_n-\left(\X_nu\right)v
\]
be the population vector minus the contribution along the eigenvector $v$.
For later convenience, we set $\Delta_{n,k}=\X_{n+k}u-\X_nu$.

\begin{lemma}\label{lemmeineq}
Let us assume (A1) and (A2). There exist $c_2' \geq 0$ and $d_2'>0$ such that for all integers $n,k \geq 1$ and for all $\varepsilon>0$, 
\begin{align}
\label{moment1}
& \left|\esp\left(\Delta_{n,k} \big|\F\right) - c_1k\left(\X_nu\right)^{\alpha}\right|
 \leq c_2' \| \Y_n \|^{\alpha}+o\left(\left(\X_nu\right)^\alpha\right),\\
\label{moment2}
& \left|\esp\left(\Delta_{n,k}^2 \big|\F\right) - kd_1\left(\X_nu\right)^{1+\alpha} \right|
 \leq d_2'\|\Y_n\|^{1+\alpha}+o\left(\left(\X_nu\right)^{1+\alpha}\right),\\
\label{moment2bis}
& \, \esp\left(|\Delta_{n,k}|^2\mathds{1}_{\left\{\Delta_{n,k} \geq \varepsilon \X_nu\right\}} \big|\F\right)
= \mathcal{O}\left( \left(\X_nu\right)^{1+\alpha+\frac{\alpha-1}{2}\delta}\right),
\end{align}
with $c_2'=0$ if $\alpha\leq 0$.
\end{lemma}

The proof of this lemma is based upon two technical lemmas that we
state and prove first.

\begin{lemma}\label{X}
Let us assume (A1) and (A2). For all $\alpha \in \left]-1,1\right[$ and $n,i \in \N$, 
\[
\esp \left(\left(\X_{n+i}u\right)^{\alpha} \big| \F \right)=\left(\X_{n}u\right)^{\alpha}+o\left(\left(\X_{n}u\right)^{\alpha}\right),
\]
and 
\[
 \esp \left(\left(\X_{n+i}u\right)^{1+\alpha} \big| \F \right)=\left(\X_{n}u\right)^{1+\alpha}+o\left(\left(\X_{n}u\right)^{1+\alpha}\right).
\]

\end{lemma}
\begin{proof}
We first prove that 
\[
\esp \left(\left(\X_{n+i}u\right)^{\gamma} \big| \F \right)=\left(\X_{n}u\right)^{\gamma}+o\left(\left(\X_{n}u\right)^{\gamma}\right),
\]
for all $\gamma \in [0,2[$ whatever the value of $\alpha$.\\
The result is obvious if $\gamma=0$. We first deal with the case where $0<\gamma \leq 1$. Then for all positive real $r$, $\left(1+r\right)^\gamma\leq 1+\gamma r$, we obtain the upper bound
\begin{align*}
 \esp \left(\left(\X_{n+1}u\right)^{\gamma} \big| \F \right) & \leq \esp \left(\X_{n+1}u \big| \F \right)^{\gamma}
\leq \left(\X_nu+g\left(\X_n\right)u\right)^\gamma\\
& \leq \left(\X_nu\right)^\gamma+\gamma g\left(\X_n\right)u \left(\X_nu\right)^{\gamma-1}.
\end{align*}
By using the inequality $\left(1+r\right)^\gamma \geq 1-|r|^\gamma$, that holds for all $r\geq -1$, we obtain the lower bound
\begin{align*}
 \esp \left(\left(\X_{n+1}u\right)^{\gamma} \big| \F \right) & \geq \esp\left( \left(\X_nu\right)^\gamma- \left|g\left(\X_n\right)u+\xi_nu\right|^\gamma \big| \F \right)\\
& \geq \left(\X_nu\right)^\gamma -2^\gamma \left(g\left(\X_n\right)u\right)^\gamma-2^\gamma \esp\left(\left|\xi_nu\right|^\gamma \big| \F\right).
\end{align*}
Since $ g\left(\X_n\right)u=\mathcal{O}\left(\left(\X_nu\right)^\alpha\right)$
and using
\[
\esp\left(\left|\xi_nu\right|^\gamma \big| \F\right)\leq \esp\left(\left|\xi_nu\right|^2 \big| \F\right)^{\frac{\gamma}{2}} = \mathcal{O}\left(\left(\X_nu\right)^{\frac{\left(1+\alpha\right)\gamma}{2}}\right)
\]
we get 
\[
\esp \left(\left(\X_{n+1}u\right)^{\gamma} \big| \F \right) = \left(\X_nu\right)^\gamma+ o\left( \left(\X_nu\right)^\gamma \right). 
\]
We now deal with the case where $\gamma>1$. Since for all real $r \geq -1$
\[
 \left(1+r\right)^\gamma \leq 1+2^{\gamma-1}|r|^\gamma+2^\gamma |r|,
\]
we obtain
\begin{align*}
 \esp \left(\left(\X_{n+1}u\right)^{\gamma} \big| \F \right) & \leq \left(\X_nu\right)^\gamma + 2^{\gamma-1}\esp\left(\left|g\left(\X_n\right)u+\xi_nu\right|^\gamma \big| \F \right)\\
& \quad + 2^\gamma \left(\X_nu\right)^{\gamma-1} \esp\left(\left|g\left(\X_n\right)u+\xi_nu\right| \big| \F \right)\\
& \leq \left(\X_nu\right)^\gamma +\mathcal{O} \left(\left(\X_nu\right)^{\frac{\left(1+\alpha\right)\gamma}{2}}\right)+\mathcal{O}\left(\left(\X_nu\right)^{\gamma+\frac{\alpha-1}{2}}\right).
\end{align*}
The lower bound is an easy consequence of Jensen's inequality:
\[
 \esp \left(\left(\X_{n+1}u\right)^{\gamma} \big| \F \right) \geq \esp \left(\X_{n+1}u \big| \F \right)^{\gamma} \geq  \left(\X_{n}u\right)^{\gamma}.
\]
We have proved that 
\begin{equation}\label{littleo}
 \esp \left(\left(\X_{n+1}u\right)^{\gamma} \big| \F \right)= \left(\X_{n}u\right)^{\gamma}+ o\left(\left(\X_{n}u\right)^{\gamma}\right),
\end{equation}
for $\gamma \in [0,2[$.
We will prove that 
\begin{equation}\label{induc}
\esp\left(f(\X_{n+1}u)\big| \F\right)=o\left(\left(\X_nu\right)^\gamma\right), 
\end{equation}
if $f$ is a real-valued function such that $f(r)=o\left(r^\gamma\right)$ when $r$ tends to infinity.
We recall that $f(r)=o\left(r^\gamma\right)$ if and only if for all $\varepsilon>0$ there exists $C_\varepsilon>0$ such that $|f(r)|\leq \varepsilon r^\gamma+C_\varepsilon$ because $\gamma>0$.\\
Let $f$ be a real-valued function such that $f(r)=o(r^\gamma)$, $\varepsilon>0$ and $C_\varepsilon>0$ such that $|f(r)|\leq \varepsilon r^\gamma+C_\varepsilon$. By \eqref{littleo},
\begin{align*}
 \esp(|f(\X_{n+1}u)|\big|\F) & \leq  \esp(\varepsilon(\X_{n+1}u)^\gamma+C_\varepsilon \big|\F)\\
& \leq \varepsilon(\X_{n}u)^\gamma+C_\varepsilon+\varepsilon o((\X_nu)^\gamma)\\
& \leq 2\varepsilon(\X_{n}u)^\gamma+C_\varepsilon+C_1.
\end{align*}
Thus, we obtain \eqref{induc}.
Since we get \eqref{littleo} and \eqref{induc}, the result follows by induction.

We end the proof with the case $-1<\alpha<0$.
The lower bound is again a consequence of Jensen's inequality: 
\begin{align*}
 \esp \left(\left(\X_{n+1}u\right)^{\alpha} \big| \F \right) & \geq \left(\X_nu+g(\X_n)u\right)^\alpha\\
& \geq \left(\X_nu\right)^\alpha+\alpha g(\X_n)u\left(\X_nu\right)^{\alpha-1}.
\end{align*}
For the upper bound, we first majorize the probability that $\X_{n+1}u$ is smaller than $\frac{\X_nu}{2}$ by Markov's inequality: 
\begin{align*}
 \pr\left(\X_{n+1}u\leq \frac{\X_nu}{2}\big| \F\right) & = \pr\left(\xi_nu\leq -\frac{\X_nu}{2}-g(\X_n)u\big| \F\right)\\
& \leq \pr\left(\xi_nu\leq -\frac{\X_nu}{2}\big| \F\right)\leq \pr\left((\xi_nu)^2 \geq \frac{(\X_nu)^2}{4}\big| \F\right)\\
& \leq \frac{4\esp\left((\xi_nu)^2 |\F\right)}{(\X_nu)^2}\leq K (\X_nu)^{\alpha-1}.
\end{align*}
Therefore, since for all $r>-\frac{1}{2}$, $(1+r)^\alpha \leq 1+4^{-\alpha}|r|$, we obtain
\begin{align*}
\esp \left(\left(\X_{n+1}u\right)^{\alpha} \big| \F \right) & = \esp \left((\mathds{1}_{\{\X_{n+1}u\leq \frac{\X_nu}{2}\}}+\mathds{1}_{\{\X_{n+1}u> \frac{\X_nu}{2}\}})\left(\X_{n+1}u\right)^{\alpha} \big| \F \right) \\
& \leq \pr\left(\X_{n+1}u \leq \frac{\X_nu}{2}\big| \F\right)\\
& \quad +\esp\left((\X_nu)^\alpha \left(1+4^{-\alpha} \frac{|g(\X_n)u+\xi_nu|}{\X_nu}\right)\big|\F\right)\\
& \leq (\X_nu)^\alpha+\mathcal{O}((\X_nu)^{\alpha-1})+\mathcal{O}((\X_nu)^{\frac{3\alpha-1}{2}}).
\end{align*}
We conclude in the same way as above by using that $f(r)=o\left(r^\alpha\right)$ if and only if for all $\varepsilon>0$ there exists $C_\varepsilon>0$ such that $|f(r)|\leq \varepsilon r^\alpha+C_\varepsilon r^{\frac{-1+\alpha}{2}}$.
\end{proof}

\begin{lemma}\label{Y}
 Let us assume (A1) and (A2). For all $\gamma \in [0,2]$  there exists $C \geq 0$ such that for all $k,n \in \N$
\[
 \sum_{i=0}^{k-1} \esp\left(\|\Y_{n+i}\|^\gamma \big|\F\right)
\leq C\|\Y_n\|^\gamma+o\left(\left(\X_nu\right)^\gamma\right).
\]
\end{lemma}

\begin{proof}
 We first write a recurrence relation for $\Y_n$ by \eqref{rec}: 
\begin{align*}
\Y_{n+1}& =\Y_nM+g\left(\X_n\right)\left(\mathrm{I}-uv\right)+\xi_n\left(\mathrm{I}-uv\right)\\
& =\Y_n\left(M-uv\right)+g\left(\X_n\right)\left(\mathrm{I}-uv\right)+\xi_n\left(\mathrm{I}-uv\right)
\end{align*}
because $\Y_nu=0$.
Hence 
\[
\Y_{n+i}=\Y_n\left(M-uv\right)^i+\sum_{j=0}^{i-1}\left(g\left(\X_{n+j}\right)+\xi_{n+j}\right)\left(\mathrm{I}-uv\right)\left(M-uv\right)^{i-1-j}.
\]
The Perron-Frobenius theorem states that the spectral radius $\lambda$ of $M-uv$ is less than $1$.  Let $\lambda_1 \in \left(\lambda,1\right)$. We recall that by Gelfand's formula (see \cite[p. 349]{horn}), we have
\[
 \lambda = \underset{m \rightarrow \infty}{\lim} \|(M-uv)^m\|^{1/m},
\]
for any matrix norm. Consequently, there exists a constant $C$ which does not depend on $i$ such that we obtain
\[
\| \Y_{n+i}\|\leq C\left(\lambda_1^i\|\Y_n\|+\sum_{j=0}^{i-1}
\left(\|g\left(\X_{n+j}\right)\|+\|\xi_{n+j}\|\right)\right).
\]
Hence, by \eqref{basicineq} below, we obtain
\[
 \| \Y_{n+i}\|^\gamma \leq C2^\gamma\lambda_1^{\gamma i}\|\Y_n\|^\gamma+C2^\gamma\sum_{j=0}^{i-1} i^\gamma
\left(\|g\left(\X_{n+j}\right)\|^\gamma+\|\xi_{n+j}\|^\gamma\right).
\]
Since
\[
\esp\left(\|g\left(\X_{n+j}\right)\|^\gamma \big| \F \right)=\mathcal{O}\left(\left(\X_nu\right)^{\alpha \gamma}\right)
\;\text{and}\;
 \esp\left(\|\xi_{n+j}\|^\gamma \big| \F \right)=\mathcal{O}\left(\left(\X_nu\right)^{\frac{\left(1+\alpha\right)\gamma}{2}}\right),
\]
by \eqref{moment2+}, we get by summation
\[
 \sum_{i=0}^{k-1} \esp\left(\|\Y_{n+i}\|^\gamma \big|\F\right)
\leq C2^\gamma \frac{1}{1-\lambda_1^\gamma}\|\Y_n\|^\gamma+o\left(\left(\X_nu\right)^\gamma\right).
\]

\end{proof}

\begin{proof}[Proof of Lemma \ref{lemmeineq}]
We first prove \eqref{moment1}: \\
The proof is an easy consequence of Lemma \ref{X}, Lemma \ref{Y} and \eqref{ineq}:
\begin{align*}
\esp\left(\Delta_{n,k} \big|\F\right) & \leq
\esp\left( \sum_{i=0}^{k-1}g\left(\X_{n+i}\right)u \big|\F\right) \\
& \leq \esp\left( \sum_{i=0}^{k-1}c_1\left(\X_{n+i}u\right)^\alpha+c_2\|\Y_{n+i}\|^\alpha+f_1\left(\X_{n+i}\right) \big|\F\right)\\
& \leq kc_1\left(\X_nu\right)^\alpha+C'\|\Y_{n}\|^\alpha+o\left(\left(\X_{n}u\right)^\alpha \right).
\end{align*}
The same proof with $-c_2$ instead of $c_2$ gives the lower bound.

We now prove inequality \eqref{moment2}. As for \eqref{moment1}, the main point is to show that $d_2'$ does not depend on $k$.

By means of Lemma \ref{X}, Lemma \ref{Y} and \eqref{ineq}, we get
\begin{align*}
\esp\left(|\Delta_{n,k}|^2 \big|\F\right)& \leq \esp\left(\left(\sum_{i=0}^{k-1}\left\{g\left(\X_{n+i}\right)u+\xi_{n+i}u\right\}\right)^2\big|\F\right)\\
& \leq \esp\left(\sum_{i=0}^{k-1}\left(\xi_{n+i}u\right)^2+\left(\sum_{i=0}^{k-1}g\left(\X_{n+i}\right)u\right)^2 \big| \F \right)\\
& +2 \esp \left(\left(\sum_{i=0}^{k-1}\xi_{n+i}u\right)\left(\sum_{i=0}^{k-1}g\left(\X_{n+i}\right)u\right)\big|\F\right)\\
& \leq kd_1\left(\X_nu\right)^{1+\alpha}+ d_2\|\Y_n\|^{1+\alpha}+o\left(\left(\X_nu\right)^{1+\alpha}\right)\\
& +\mathcal{O}\left(\left(\X_nu\right)^{2\alpha}\right)+\mathcal{O}\left(\left(\X_nu\right)^{\frac{\alpha\left(1+\alpha\right)}{2}}\right),
\end{align*}
and the proof for the lower bound is similar.
We conclude with the proof of  \eqref{moment2bis}.\\
By Markov's inequality: 
 \begin{align*}
\esp&\left(|\Delta_{n,k}|^2\mathds{1}_{\left\{\Delta_{n,k} \geq \varepsilon \X_nu\right\}} \big|\F\right) \\
& \leq \esp\left(|\Delta_{n,k}|^2\mathds{1}_{\left\{\left(\Delta_{n,k}\right)^\delta \geq \left(\varepsilon \X_nu\right)^\delta\right\}} \big|\F\right)\\
& \leq \esp\left(\frac{|\Delta_{n,k}|^{2+\delta}}{\left(\varepsilon \X_nu\right)^\delta}\big|\F\right)\\
& \leq \frac{\left(2k\right)^{2+\delta}}{\left(\varepsilon \X_nu\right)^\delta} \esp\left(\sum_{i=0}^{k-1}\left|g\left(\X_{n+i}\right)u\right|^{2+\delta}+\left|\xi_{n+i}u\right|^{2+\delta}\big|\F\right).
\end{align*}
Since $\esp\left(\left|g\left(\X_{n+i}\right)u\right|^{2+\delta}\big| \F\right)=\mathcal{O}\left(\left(\X_nu\right)^{2\alpha+\alpha\delta}\right)$ by Lemma \ref{X} and \\$\esp\left(\left|\xi_{n+i}u\right|^{2+\delta}\Big|\F\right)=\mathcal{O}\left(\left(\X_nu\right)^{1+\alpha+\frac{1+\alpha}{2}\delta}\right)$ by \eqref{moment2+}, we obtain
\begin{align*}
\esp&\left(|\Delta_{n,k}|^2\mathds{1}_{\left\{\Delta_{n,k} \geq \varepsilon \X_nu\right\}} \big|\F\right) \\
& \leq \frac{\left(2k\right)^{2+\delta}}{\left(\varepsilon \X_nu\right)^\delta}\left(e_1\left(k\right) \left(\X_nu\right)^{2\alpha+\alpha\delta}+e_2\left(k\right)\left(\X_nu\right)^{1+\alpha+\frac{1+\alpha}{2}\delta}\right)\\
& \leq e_1'\left(k, \varepsilon \right) \left(\X_nu\right)^{1+\alpha+\frac{\alpha-1}{2}\delta},
\end{align*}
which is the desired inequality.
\end{proof}

We now prove Lemmas \ref{loglyapu} and \ref{1/loglyapu}.
\begin{proof}[Proof of Lemma \ref{loglyapu}]
 We first recall an inequality proved in \cite{kersting}: 
If $\varepsilon>0$, $x>0$ and $h>-x$, then
\begin{equation}\label{kerst1}
\log\left(x+h\right)\leq \log x +\frac{h}{x}-\frac{h^2 \mathds{1}_{\left\{h\leq \varepsilon x\right\}}}{2\left(1+\varepsilon\right)x^2}.
\end{equation}
Let $k\in \N$ and $\varepsilon>0$, both to be fixed later on. 
We apply inequality \eqref{kerst1} with $x=\X_nu$ and $h=\Delta_{n,k}$:
\begin{align*}
& \esp\left(\log\left(\X_{n+k}u\right)\big|\F\right)\leq \\
&  \log\left(\X_nu\right)+\frac{\esp\left(\Delta_{n,k}\big|\F\right)}{\X_nu}
 -\frac{\esp\left(|\Delta_{n,k}|^2\big|\F\right)}{2\left(1+\varepsilon\right)\left(\X_nu\right)^2}
 +\frac{\esp\left(|\Delta_{n,k}|^2\mathds{1}_{\left\{\Delta_{n,k} > \varepsilon \X_nu\right\}}
 \big|\F\right)}{2\left(1+\varepsilon\right)\left(\X_nu\right)^2}.
\end{align*}
Using inequalities \eqref{moment1}, \eqref{moment2} and \eqref{moment2bis} from Lemma \ref{lemmeineq}  we obtain
\begin{align*}
& \esp\left(\log\left(\X_{n+k}u\right)\big|\F\right)\leq \\
& \log\left(\X_nu\right)+\frac{c_1k\left(\X_nu\right)^{\alpha}+c_2' \| \Y_n \|^{\alpha}+o\left(\left(\X_nu\right)^\alpha\right)}{\X_nu}\\
& -\frac{kd_1\left(\X_nu\right)^{1+\alpha}-d_2'\|\Y_n\|^{1+\alpha}+o\left(\left(\X_nu\right)^{1+\alpha}\right)}{2\left(1+\varepsilon\right)\left(\X_nu\right)^2}
+\frac{\mathcal{O}\left( \left(\X_nu\right)^{1+\alpha+\frac{\alpha-1}{2}\delta}\right)}{2\left(1+\varepsilon\right)\left(\X_nu\right)^2}.
\end{align*}
By the Perron-Frobenius Theorem \cite{seneta}, all coordinates of $u$ are positive.
Therefore, by definition of $\Y_n$, there exists $b>0$ such that for every $n$,
\begin{equation}\label{xsury}
\|\Y_n\|\leq b \X_nu.
\end{equation}
We obtain
\begin{align*}
& \esp\left(\log\left(\X_{n+k}u\right)\big|\F\right)\leq \\
& \log\left(\X_nu\right)+\frac{c_1k\left(\X_nu\right)^{\alpha}+c_2'b^\alpha \left(\X_nu\right)^{\alpha}+o\left(\left(\X_nu\right)^\alpha\right)}{\X_nu}\\
&-\frac{kd_1\left(\X_nu\right)^{1+\alpha}-d_2'b^{1+\alpha}\left(\X_nu\right)^{1+\alpha}+o\left(\left(\X_nu\right)^{1+\alpha}\right)}{2\left(1+\varepsilon\right)\left(\X_nu\right)^2}
+\frac{\mathcal{O}\left(\left(\X_nu\right)^{1+\alpha+\frac{\alpha-1}{2}\delta}\right)}{2\left(1+\varepsilon\right)\left(\X_nu\right)^2}.
\end{align*}
We first choose $\varepsilon>0$ such that $c_1<\frac{d_1}{2\left(1+\varepsilon\right)}$. We now choose $k$ such that
\[
k\left(c_1-\frac{d_1}{2\left(1+\varepsilon\right)}\right)+c_2'b^\alpha+\frac{d_2'b^{1+\alpha}}{2\left(1+\varepsilon\right)}<0.
\]
Thus there exists $s>0$ such that,
\[
\esp\left(\log\left(\X_{n+k}u\right)\big|\F\right)\leq \log\left(\X_nu\right), \text{ if } \X_nu>s.
\]
\end{proof}

\begin{proof}[Proof of Lemma \ref{1/loglyapu}]
 We recall another inequality proved in \cite{kersting}.
For $x\geq 3$, let
\[
L\left(x\right)=\left(\log x\right)^{-1}.
\]
There exists $C_2>0$ such that for any $x\geq 3$, $h>3-x$ and $0<\delta \leq 1$ then
\begin{equation}\label{kerst2}
L\left(x+h\right) \leq L\left(x\right)+L'\left(x\right)h+
\frac{L''\left(x\right)h^2}{2}+C_2\frac{|h|^{2+\delta}}{\left(\log x\right)^{2}x^{2+\delta}}+\mathds{1}_{\left\{h\leq-\frac{x}{2}\right\}}.
\end{equation}

As in the first case, we prove that $\esp\left(L\left(\X_{n+k}u\right)\big| \F \right)\leq L\left(\X_nu\right)$ for some fixed $k$ and $\X_nu$ large enough.

We apply inequality \eqref{kerst2} with $x=\X_nu$, $h=\Delta_{n,k}$ and $k$ an integer to be fixed later on to get
\begin{align*}
& \esp\left(L\left(\X_{n+k}u\right)\big| \F \right)  \leq \\
& L\left(\X_nu\right) -\frac{\esp\left(\Delta_{n,k} \big|\F\right)}{\left(\X_nu\right)\left(\log \left(\X_nu\right)\right)^{2}}
 \quad +\frac{\esp\left(|\Delta_{n,k}|^2\big|\F\right)}{2\left(\X_nu\right)^2\left(\log\left(\X_nu\right)\right)^{2}}+\frac{2\esp\left(|\Delta_{n,k}|^2 \big|\F\right)}{2\left(\X_nu\right)^2\left(\log\left(\X_nu\right)\right)^{3}}\\
& + C_2\frac{\esp\left(|\Delta_{n,k}|^{2+\delta} \big|\F\right)}{\left(\log \left(\X_nu \right)\right)^{2}\left(\X_nu\right)^{2+\delta}} + \esp\left(\mathds{1}_{\left\{\Delta_{n,k} \leq -\frac{\X_nu}{2}\right\}} \big|\F\right).
\end{align*}

We start with the estimate
\begin{align*}
\esp\left(\mathds{1}_{\left\{\Delta_{n,k} \leq -\frac{\X_nu}{2}\right\}} \big|\F\right) & \leq \esp\left(\mathds{1}_{\left\{2^{2+\delta}\frac{\left|\Delta_{n,k}\right|^{2+\delta}}{\left(\X_nu\right)^{2+\delta}} \geq 1\right\}} \big|\F\right)\\
& \leq \esp\left(2^{2+\delta}\frac{\left|\Delta_{n,k}\right|^{2+\delta}}{\left(\X_nu\right)^{2+\delta}}\big|\F\right),
\end{align*}
that follows easily from Markov's inequality.
We now use the basic inequality 
\begin{equation}\label{basicineq}
\left(a+b\right)^{2+\delta} \leq 2^{2+\delta} \left(a^{2+\delta}+b^{2+\delta} \right), \, a,b>0,
\end{equation}
and the facts (resulting from \eqref{ineq} and \eqref{moment2+}) that there exist some positive real numbers $A$ and $B$ such that 
\[
\esp \left( \left|g \left(\X_{n+i}\right)u\right|^{2+\delta} \big| \F \right)\leq A \left(\X_nu\right)^{\alpha\left(2+\delta\right)},
\]
and 
\[
 \esp\left(\left|\xi_{n+i}u\right|^{2+\delta} \big| \F \right)\leq B \left(\X_nu\right)^{\left(\frac{\alpha+1}{2}\right)\left(2+\delta\right)},
\]
to obtain the upper bound
\begin{align*}
 \esp\left(\left|\Delta_{n,k}\right|^{2+\delta} \big| \F \right)& \leq (2k)^{2+\delta} \esp\left(\sum_{i=0}^{k-1}\left|g\left(\X_{n+i}\right)u\right|^{2+\delta}+\left|\xi_{n+i}u\right|^{2+\delta}\big| \F \right)\\
& \leq C_3(k) \left(\X_nu\right)^{\left(\frac{\alpha+1}{2}\right)\left(2+\delta\right)}.
\end{align*}

Therefore, there exists $C_4(k)$ such that
\[
\esp\left(\mathds{1}_{\left\{\Delta_{n,k} \leq -\frac{\X_nu}{2}\right\}} \big|\F\right) \leq C_4(k) \left(\X_nu\right)^{\left(\frac{\alpha-1}{2}\right)\left(2+\delta\right)}. 
\]

We use the inequalities \eqref{moment1}, \eqref{moment2} and \eqref{moment2bis} from Lemma \ref{lemmeineq} and inequality \eqref{xsury}: 
\begin{align*}
& \esp\left(L\left(\X_{n+k}u\right)\big| \F \right)  \\
& \leq  L\left(\X_nu\right)  -\frac{\left(c_1k\left(\X_nu\right)^{\alpha}-c_2' \| \Y_n \|^{\alpha} +o\left(\left(\X_nu\right)^\alpha\right)\right)}{\left(\X_nu\right)\left(\log \left(\X_nu\right)\right)^{2}}\\
& \quad + \frac{\left(kd_1\left(\X_nu\right)^{1+\alpha} +d_2'\|\Y_n\|^{1+\alpha} +o\left(\left(\X_nu\right)^{1+\alpha}\right)\right)}{2\left(\X_nu\right)^2\left(\log \left(\X_nu\right)\right)^{2}}\\
& \quad + \frac{2\left(kd_1\left(\X_nu\right)^{1+\alpha}+d_2'\|\Y_n\|^{1+\alpha}+o\left(\left(\X_nu\right)^{1+\alpha}\right)\right)}{2\left(\X_nu\right)^2\left(\log \left(\X_nu \right)\right)^{3}}\\
&  \quad +C_2\frac{\mathcal{O}\left(\left(\X_nu\right)^{1+\alpha+\frac{\alpha-1}{2}\delta}\right)}{\left(\log \left(\X_nu \right) \right)^{2}\left(\X_nu\right)^{2+\delta}}+C_4(k) \left(\X_nu\right)^{\frac{\left(\alpha-1\right)\left(2+\delta\right)}{2}}\\
& \leq    L\left(\X_nu\right) +\frac{k\left(\frac{d_1}{2}-c_1\right)+b_2'}{\left(\X_nu\right)^{1-\alpha}\left(\log \left(\X_nu \right)\right)^{2}}
\quad +o\left(\frac{1}{\left(\X_nu\right)^{1-\alpha}\left(\log\X_nu\right)^{2}}\right).
\end{align*}
Since $d_1/2<c_1$, we first choose $k$ such that 
\[
k\left(\frac{d_1}{2}-c_1\right)+b_2'<0,
\]
with $b_2'=b^{1+\alpha}d_2'/2+b^\alpha c_2'$.
Then there exists $s>0$ such that, 
\[
\esp\left(L\left(\X_{n+k}u\right) \big| \F \right)  \leq L\left(\X_nu\right), \text{ if } \X_nu>s.
\]
\end{proof}

\noindent \textbf{Acknowledgement.}  \\
The author thanks Vincent Bansaye and Jean-Ren\'e Chazottes for many helpful discussions on the subject of this paper. He is also very grateful to the referee for his careful reading of the original manuscript.
This article benefited from the support of the ANR MANEGE (ANR-09-BLAN-0215) and from the Chair ``Mod\'elisation Math\'ematique et Biodiversit\'e'' of Veolia Environnement - Ecole Polytechnique - Museum National d'Histoire Naturelle - Fondation X.
\bibliographystyle{plain}
\bibliography{synthese}
\end{document}